\documentclass[11pt]{amsart}
\usepackage{amsfonts,amssymb,amsthm,eucal,color,bbm,verbatim,enumitem}

\usepackage[margin=1.25in]{geometry}

\newtheorem{thm}{Theorem}
\newtheorem{cor}[thm]{Corollary}
\newtheorem{lemma}[thm]{Lemma}
\newtheorem{prop}[thm]{Proposition}

\newcommand{\R}{\mathbb{R}}
\newcommand{\E}{\mathbb{E}}

\newcommand{\Prob}{\mathbb{P}}
\renewcommand{\P}{\mathbb{P}}
\newcommand{\N}{\mathbb{N}}
\newcommand{\Z}{\mathbb{Z}}
\newcommand{\C}{\mathbb{C}}

\DeclareMathOperator{\tr}{tr}

\newcommand{\abs}[1]{\left\vert #1 \right\vert}
\newcommand{\norm}[1]{\left\Vert #1 \right\Vert}

\newcommand{\eps}{\varepsilon}

\DeclareMathOperator{\Lip}{Lip}

\DeclareMathOperator{\sgn}{sgn}

\newcommand{\Unitary}[1]{\mathbb{U}(#1)}
\newcommand{\SUnitary}[1]{\mathbb{SU}(#1)}
\newcommand{\Orthogonal}[1]{\mathbb{O}(#1)}
\newcommand{\SOrthogonal}[1]{\mathbb{SO}(#1)}
\newcommand{\Symplectic}[1]{\mathbb{S}\mathbbm{p}(2 #1)}
\newcommand{\Circle}{\mathbb{S}^1}
\newcommand{\SOneg}[1]{\mathbb{SO}^{-}(#1)}

\allowdisplaybreaks[3]
\numberwithin{thm}{section}
\numberwithin{equation}{section}

\author{Elizabeth S.\ Meckes and Mark W.\ Meckes}

\address{Department of Mathematics, Case Western Reserve University,
10900 Euclid Ave., Cleveland, Ohio 44106, U.S.A.}

\email{elizabeth.meckes@case.edu}

\address{Department of Mathematics, Case Western Reserve University,
10900 Euclid Ave., Cleveland, Ohio 44106, U.S.A.}

\email{mark.meckes@case.edu}

\title[Spectral measures of random matrices]{Concentration and
  convergence rates \\ for spectral measures of random matrices}

\begin{document}

\begin{abstract}
  The topic of this paper is the typical behavior of the spectral
  measures of large random matrices drawn from several 
  ensembles of interest, including in particular matrices drawn from
  Haar measure on the classical Lie groups, random compressions of
  random Hermitian matrices, and the so-called random sum of two
  independent random matrices.  In each case, we estimate the expected
  Wasserstein distance from the empirical spectral measure to a
  deterministic reference measure, and prove a concentration result
  for that distance.  As a consequence we obtain almost sure
  convergence of the empirical spectral measures in all cases.
\end{abstract}

\maketitle


\section{Introduction}

The topic of this paper is the typical behavior of the spectral
measures of large random matrices drawn from several 
ensembles of interest.  Specifically, we consider random matrices
drawn from Haar measure on the classical Lie groups $\Orthogonal{n}$,
$\SOrthogonal{n}$, $\Unitary{n}$, $\SUnitary{n}$, and
$\Symplectic{n}$; Dyson's circular ensembles; random compressions of
random Hermitian matrices satisfying a concentration hypothesis
(including random Wigner matrices as a special case); and a random
matrix model considered in free probability described by the sum of
two random Hermitian matrices, one of which has been subjected to a
random basis change.  In each case, we estimate the expected
Wasserstein distance from the empirical spectral measure to a
deterministic reference measure, and prove a concentration result for
that distance.  Our bounds are sufficient to obtain almost sure
convergence of the empirical spectral measures (with rates in the
Wasserstein distance) in all cases.

The proofs follow the same approach as the recent work of E.\ Meckes
\cite{EMeckes2} on random projections of high-dimensional probability
measures.  The central idea is to view the Wasserstein distance
$d_1(\mu_M,\mu)$ from the empirical spectral measure of a random
matrix $M$ to a deterministic reference measure $\mu$ as the supremum
of a stochastic process indexed by the unit ball of the
(infinite-dimensional) space $\Lip(\C)$ of real-valued Lipschitz
functions on $\C$.  Concentration properties of the random matrices
considered imply that the stochastic process in question satisfies a
subgaussian increment condition; Dudley's entropy bound together with
approximation arguments are then used to bound the expected supremum
of the process.  In the case of the classical Lie groups, earlier work
by Diaconis and Mallows \cite{diamal}, Diaconis and Shahshahani
\cite{DiSh}, and Rains \cite{Rains-low-powers} is used to show that
the deterministic reference measure can be taken to be the uniform
measure on the circle, and the classical measure concentration results
of Gromov and Milman \cite{GrMi} are used to obtain the needed
concentration properties.  For the Hermitian models, the deterministic
reference measure used is simply the average of the empirical spectral
measure and the matrices are assumed at the outset to satisfy a
concentration hypothesis.

Further history and motivation are discussed in sections
\ref{S:lie-groups} and \ref{S:Hermitian} below; the remainder of this
section is devoted to notation and conventions.

\medskip

For a subset $A \subseteq \C$, the space of Lipschitz functions $f: A
\to \R$ is denoted by $\Lip(A)$, and is equipped with the Lipschitz
seminorm $\abs{\cdot}_{\Lip}$.  Denote by $\mathcal{P}(A)$ the space of
all probability measures supported in $A$, and by $\mathcal{P}_p(A)$
be the space of probability measures in $\mathcal{P}$ with finite
$p$th moment, equipped with the $L_p$ Wasserstein distance $d_p$
defined by
\begin{equation}\label{D:distance_coupling}
  d_p(\mu,\nu) := \inf_{\pi}
  \left(\int \abs{x-y}^p \ d\pi(x,y)\right)^{1/p}.
\end{equation}
The infimum above is over probability measures $\pi$ on $A \times A$
with marginals $\mu$ and $\nu$.  Note that $d_p \le d_q$ when $p \le
q$. The $L_1$ Wasserstein distance can be equivalently defined (see,
e.g., \cite{Du}) by
\begin{equation}\label{D:distance_testfcns}
d_1(\mu,\nu):=\sup_f\int \bigl[f(x)-f(y)\bigr] \ d\mu(x)d\nu(y),
\end{equation}
where the supremum is over $f$ in the unit ball $B(\Lip(A))$ of
$\Lip(A)$. In what follows, ``Wasserstein distance'' with $p$
unspecified refers to $d_1$.

Denote by $\mathcal{M}_n^{sa}$ the space of $n \times n$ Hermitian
matrices, by $\mathcal{N}_n$ the space of $n \times n$ normal
matrices. Denote by $\Unitary{n}$ the group of $n \times n$ unitary
matrices, by $\Orthogonal{n}$ the group of $n \times n$ real
orthogonal matrices, by $\SUnitary{n}$ and $\SOrthogonal{n}$
respectively the special unitary and orthonal groups, and by
$\Symplectic{n} \subseteq \Unitary{2n}$ the compact symplectic group.
In all results below these are understood to be equipped with the
Hilbert--Schmidt norm $\norm{\cdot}_{HS}$.  For any $A\in
\mathcal{N}_n$, let $\mu_A$ denote the spectral distribution of $A$;
that is, if $\{\lambda_i\}_{i=1}^n$ are the eigenvalues of $A$, then
$\mu_A:=\frac{1}{n}\sum_{i=1}^n\delta_{\lambda_i}$.

For $A \in \mathcal{M}_n^{sa}$, denote by $\delta(A) :=
\lambda_{\max}(A) - \lambda_{\min}(A)$ the \emph{spectral diameter} of
$A$.  Note in particular that
\[
\delta(A) = 2 \inf_{\lambda \in \R} \norm{A - \lambda I}_{op},
\]
where $\norm{\cdot}_{op}$ denotes the operator norm.

Throughout Sections \ref{S:lie-groups} and \ref{S:Hermitian}, $c$,
$C$, and similar symbols denote absolute positive constants, whose
exact values may vary from one instance to another.


\bigskip

\section{Random matrices in classical Lie groups}\label{S:lie-groups}

This section is concerned primarily with a random matrix $U$ drawn
according to Haar measure from one of the classical compact Lie groups
$\Orthogonal{n}$, $\SOrthogonal{n}$, $\Unitary{n}$, $\SUnitary{n}$,
and $\Symplectic{n}$.  It will be shown (see Corollary
\ref{T:group-concentration-distance} below) that for fixed $n$, the
empirical spectral measure $\mu_U$ is tightly concentrated near the
uniform measure $\nu$ on $\Circle = \{ z \in \C : \abs{z} = 1\}$, with
mean Wasserstein distance of order at most $n^{-2/3}$ and subgaussian
tail bounds.  As a consequence, it is shown (see Corollary \ref{T:group-BC})
that the Wasserstein distance between $\mu_U$ and $\nu$ is almost
surely of order at most $n^{-2/3}$.  We do not claim that these
results are sharp; in fact, there is reason to suspect that $n^{-2/3}$
could be replaced by $n^{-1}$, up to logarithmic factors.  However, to
the best of our knowledge these are the first results which achieve
any bounds for these quantities.

Random matrices from these groups have been extensively studied, and
much is already known.  In particular, we use results from
\cite{diamal}, \cite{DiSh}, and \cite{Rains-low-powers} below in order
to show that the uniform distribution on the circle is the correct
reference measure for these ensembles.  In the case of the unitary and
special unitary groups $\Unitary{n}$ and $\SUnitary{n}$, large
deviations principles for the empirical spectral measures have been
proved by Hiai and Petz \cite{HiPe-ld} and Hiai, Petz, and Ueda
\cite{HiPeUe}, respectively.  The rates in those LDPs are consistent
with the level of concentration we obtain for the distance, and both
results imply in particular the almost sure convergence of the
spectral measures, although the LDPs do not give information about the
rates of convergence.  It should be noted that almost sure convergence
for random unitary matrices was proved prior to the results of Hiai and
Petz in Voiculescu's paper \cite{Vo}.
As far as we know, almost sure convergence for
the spectral distributions of matrices from the other groups above was
not previously known.

The approach taken in this section has three main steps:
\begin{enumerate}
\item The mean ESD $\mu = \E \mu_U$ approximates $\nu$ in Wasserstein
  distance (Theorem \ref{T:group-avg-distances}). This is shown using
  known moments of $\mu$ and classical results on approximating
  Lipschitz functions on $\Circle$ by polynomials.
\item The mean Wasserstein distance $\E d_1(\mu_U, \mu)$ is small
  (Theorem \ref{T:group-entropy}).  Using definition
  \eqref{D:distance_testfcns}, the Wasserstein distance is interpreted
  as the supremum of a stochastic process indexed by test functions.
  Concentration of measure on the classical Lie groups implies that
  this process has subgaussian increments, allowing the expected
  supremum to be estimated via entropy methods.
\item The Wasserstein distance $d_1(\mu_U, \nu)$ is tightly
  concentrated near its mean (Theorem
  \ref{T:group-concentration-distance}), and almost sure convergence
  of $\mu_U$ --- with the indicated rate in Wasserstein distance ---
  follows from the Borel--Cantelli lemma (Corollary \ref{T:group-BC}).
  This concentration is again shown using concentration of measure on
  the classical Lie groups.
\end{enumerate}

In contrast to the proofs of the LDPs in \cite{HiPe-ld,HiPeUe}, the
proofs here make no use of the joint densities of eigenvalues in the
classical Lie groups.

There is an important technical caveat to the strategy outlined above,
which is that the general concentration of measure results known for
$\SOrthogonal{n}$, $\SUnitary{n}$, and $\Symplectic{n}$ do not extend
to $\Orthogonal{n}$ and $\Unitary{n}$.  The latter two cases will
instead be handled basically by reducing to the corresponding special
groups.  For this purpose it will be useful also to consider Haar
measure on the coset $\SOneg{n} = \{ U \in \Orthogonal{n} : \det U =
-1\}$.  (In this case Haar measure refers to invariance under the
action of $\SOrthogonal{n}$.)

The same strategy can also be carried out for random matrices from
Dyson's Circular Ensembles, as indicated in Theorem \ref{T:circular}.

\medskip

The first step of the plan of this section is achieved in the
following theorem.  Here and in the following, $U \in G$ means that
$U$ is distributed according to Haar measure on the group (or coset)
$G$. Recall that $\nu$ denotes the uniform probability measure on
$\Circle$, and that $\mu = \E \mu_U$.

\begin{thm}\label{T:group-avg-distances}
  \begin{enumerate}
  \item If $U \in \Unitary{n}$ then $\mu = \nu$.

  \item If $U \in \SUnitary{n}$ then $d_1 (\mu, \nu) \le \frac{C}{n}$.

  \item If $U \in \SOrthogonal{n}$, $\Orthogonal{n}$, $\SOneg{n}$, or
    $\Symplectic{n}$, then $d_1(\mu, \nu) \le C \frac{\log n}{n}$.
  \end{enumerate}
\end{thm}

\begin{proof}
  \begin{enumerate}
  \item For any fixed $\omega \in \Circle$, $\omega U$ is also
    Haar-distributed in $\Unitary{n}$. Therefore $\mu$ is a
    rotation-invariant probability measure on $\Circle$, hence equal
    to $\nu$.

  \item Observe first that $e^{2\pi i /n} I_n \in \SUnitary{n}$, and
    so $e^{2\pi i /n}U$ is Haar-distributed in $\SUnitary{n}$. Thus
    for any integer $k$,
    \[
    \E \tr U^k = \E \tr (e^{2\pi i /n} U)^k = e^{2\pi i k/n} \E \tr U^k.
    \]
    Therefore $\E \tr U^k = 0$ for $1 \le \abs{k} < n$. If $g(z) =
    \sum_{\abs{k}<n} a_k z^k$ is a trigonometric polynomial on
    $\Circle$, it follows that
    \[
    \int g \ d \mu = \E \int g \ d\mu_U 
    = \frac{1}{n} \sum_{\abs{k}<n} a_k \E \tr U^k = a_0 = \int g \ d\nu.
    \]
    Now given $f:\Circle \to \R$ which is $1$-Lipschitz, Jackson's
    theorem (see, e.g.\ \cite[Theorem 1.4]{Rivlin}) implies that there
    is such a polynomial $g$ such that $\norm{f - g}_\infty \le
    \frac{C}{n}$. Thus
    \begin{align*}
    \abs{\int f \ d\mu - \int f \ d\nu} & \le
    \abs{\int f \ d\mu - \int g \ d\mu} 
    + \abs{\int g \ d\nu - \int f \ d\nu} \\
      & \le 2 \norm{f - g}_\infty
      \le \frac{2C}{n}.
    \end{align*}

  \item By results of Diaconis and Mallows (see \cite{diamal}),
    Diaconis and Shahshahani \cite{DiSh}, and Rains
    \cite{Rains-low-powers}, in each of these cases $\abs{\E \tr U^k}
    \le 1$ for $1 \le \abs{k} < n$.

    Given $f:\Circle \to \R$ which is $1$-Lipschitz, it is easy to
    check that $\abs{\widehat{f}(k)} \le \frac{C}{k}$ for $\abs{k} \ge
    1$ (see, e.g., Theorem 4.6 of \cite{Katz}).  If
    \[
    S_n(z) = \sum_{k=-(n-1)}^{n-1} \widehat{f}(k) z^k,
    \]
    then
    \begin{align*}
      \abs{\int S_n \ d\mu - \int S_n \ d\nu} & =
      \abs{\frac{1}{n} \sum_{1 \le \abs{k} \le n-1} \widehat{f}(k) 
        \E \tr U^k} 
      \le \frac{C}{n} \sum_{1 \le k \le n-1} \frac{1}{k}
      \le C \frac{\log n}{n}.
    \end{align*}
    A theorem of Lebesgue (see, e.g., \cite[Theorem 2.2]{Rivlin})
    implies that
    \[
    \norm{f - S_n}_\infty \le C (\log n) \inf_g \norm{f - g}_\infty,
    \]
    where the infimum is over all trigonometric polynomials $g(z) =
    \sum_{\abs{k}<n} a_k z^k$.  Combined with Jackson's theorem
    \cite[Theorem 1.4]{Rivlin} this implies that $\norm{f -
      S_n}_\infty \le C \frac{\log n}{n}$, and thus
    \begin{align*}
      \abs{\int f \ d\mu - \int f \ d\nu} & \le \abs{\int f \ d\mu -
        \int S_n \ d\mu}
      + \abs{\int S_n \ d\mu - \int S_n \ d\nu} \\
      & \quad + \abs{\int S_n \ d\nu - \int f d\nu} \\
      & \le C \frac{\log n}{n}. \qedhere
    \end{align*}
  \end{enumerate}
\end{proof}

\medskip

The second and third steps of the plan of this section rely on the
following concentration of measure property.  This essentially follows
from a general isoperimetric inequality for Riemannian manifolds due
to Gromov and Milman \cite{GrMi} and calculations of the Ricci
curvature of the classical Lie groups (for which see \cite[Appendix
F]{AnGuZe}). In the precise form stated it follows from a result of
Bakry and \'Emery \cite{BaEm} which shows that the same Ricci
curvature bounds imply a logarithmic Sobolev inequality, which in turn
implies such a concentration inequality (cf.\ \cite[Chapter
5]{Ledoux}).

\begin{prop}[See {\cite[Theorem 4.4.27]{AnGuZe}}] \label{T:GM} Let $G$
  be one of $\SOrthogonal{n}$, $\SOneg{n}$, $\SUnitary{n}$, or
  $\Symplectic{n}$. Let $F: G \to \R$ be $1$-Lipschitz with respect to
  the geodesic metric (induced by the standard embedding in matrix
  space with the Hilbert--Schmidt norm). If $U \in G$, then
  \[
  \Prob\bigl[F(U) - \E F(U) \ge t \bigr] \le  e^{- c n t^2}
  \]
  for every $t > 0$.
\end{prop}

The geodesic metric on $G$ dominates the Hilbert--Schmidt metric on
matrix space, so the conclusion of Proposition \ref{T:GM} applies in
particular to $F$ which is $1$-Lipschitz with respect to the
Hilbert--Schmidt metric.

The following lemma provides the necessary Lipschitz estimates for the
functions to which the concentration property will be applied in this
and the subsequent section.

\begin{lemma}\label{T:Lipschitz-estimates}
  The map $A\mapsto\mu_A$ from $\mathcal{N}_n$ to $\mathcal{P}_1(\C)$
  taking a normal matrix to its spectral measure is
  $n^{-1/2}$-Lipschitz. Furthermore, if $\rho \in \mathcal{P}_1(\C)$
  is any fixed probability measure, the following statements hold.
  \begin{enumerate}
  \item \label{X_f-lipschitz} For any $1$-Lipschitz function
    $f:\C\to\R$, 
    the function
    \[
    A \mapsto \int f \ d\mu_A - \int f \ d\rho
    \]
    is $n^{1/2}$-Lipschitz. 
  \item \label{d_1-lipschitz} The map $A \mapsto d_1(\mu_A,\rho)$ is
    $n^{1/2}$-Lipschitz.
\end{enumerate}
\end{lemma}

\begin{proof}
  If $A$ and $B$ are $n \times n$ normal matrices, then the
  Hoffman--Wielandt inequality \cite[Theorem VI.4.1]{Bhatia} implies
  that
  \begin{equation}\label{E:permutations_HS_bound}
  \min_{\sigma \in S_n} \sum_{j=1}^n \abs{\lambda_j(A) -
    \lambda_{\sigma(j)}(B)}^2 \le \norm{A - B}_{HS}^2, 
  \end{equation}
  where $\lambda_1(A), \dotsc, \lambda_n(A)$ and $\lambda_1(B),
  \dotsc, \lambda_n(B)$ are the eigenvalues (with multiplicity, in any
  order) of $A$ and $B$ respectively. Defining couplings of $\mu_A$
  and $\mu_B$ given by
  \[
  \pi_\sigma = \frac{1}{n} \sum_{j=1}^n \delta_{(\lambda_j(A),
    \lambda_{\sigma(j)}(B))}
  \]
  for $\sigma \in S_n$, it follows from \eqref{D:distance_coupling}
  and \eqref{E:permutations_HS_bound} that
  \begin{align*}
    d_1(\mu_A, \mu_B)^2 & \le d_2(\mu_A, \mu_B)^2 
    \le \min_{\sigma \in S_n} \int \abs{w - z}^2 \ d\pi_\sigma(w,z) \\
    & = \min_{\sigma \in S_n} \frac{1}{n}
    \sum_{j=1}^n \abs{\lambda_j(A) - \lambda_{\sigma(j)}(B)}^2 
     \le \frac{1}{n} \norm{A - B}_{HS}^2,
  \end{align*}
  proving the first statement of the lemma.  The final claim that $A
  \mapsto d_1(\mu_A, \rho)$ is $n^{-1/2}$-Lipschitz is now immediate.

  By the definition in \eqref{D:distance_testfcns} of $d_1$, given a
  $1$-Lipschitz $f:\C \to \R$, the mapping $\mathcal{P}_1(\C) \to \R$,
  $\mu \mapsto \int f \ d\mu$ is $1$-Lipschitz.  Combined with the
  above argument, this implies that the function
  \[
  A \mapsto \int f \ d\mu_A - \int f \ d\rho
  \]
  is $n^{-1/2}$-Lipschitz on $\mathcal{N}_n$.
\end{proof}

\begin{cor}\label{T:group-concentration}
  Let $G$ be one of $\SOrthogonal{n}$, $\SOneg{n}$, $\SUnitary{n}$, or
  $\Symplectic{n}$, and let $U \in G$.
  \begin{enumerate}
  \item For any fixed probability measure $\rho \in
    \mathcal{P}(\Circle)$ and $1$-Lipschitz $f:\Circle \to \R$, define
    the random variable
    \[
    X_f = \int f \ d\mu_U - \int f \ d\rho.
    \]
    Then
    \[
    \Prob \bigl[\abs{X_f - \E X_f} \ge t \bigr]
    \le 2 e^{-c n^2 t^2}
    \]
    for every $t > 0$.    
  \item For any fixed probability measure $\rho \in
    \mathcal{P}(\Circle)$,
    \[
    \Prob \bigl[d_1(\mu_U,\rho) - \E d_1(\mu_U,\rho) \ge t \bigr]
    \le e^{-c n^2 t^2}.
    \]
    for every $t > 0$.
  \end{enumerate}
\end{cor}

\begin{proof}
  The first part of the corollary follows from Proposition \ref{T:GM}
  and part \eqref{X_f-lipschitz} of Lemma \ref{T:Lipschitz-estimates}
  (applied to both $X_f$ and $-X_f=X_{-f}$).

  The second part of the corollary follows from Proposition \ref{T:GM}
  and part \eqref{d_1-lipschitz} of Lemma \ref{T:Lipschitz-estimates}.
\end{proof}

As noted earlier, the strategy outlined above does not apply directly
to the full unitary and orthogonal groups, due to the lack of the
concentration property of Proposition \ref{T:GM}.  The results of
Gromov--Milman and Bakry--\'Emery fail to apply to $\Orthogonal{n}$
because it is not connected, and to $\Unitary{n}$ because its Ricci
tensor is degenerate.  Nevertheless, the main results of this section
can be extended to $\Unitary{n}$ and $\Orthogonal{n}$. In the
orthogonal case this will be done by conditioning on $\det U$, which
is why it is convenient to consider also the case of random matrices
in $\SOneg{n}$. The unitary case could be handled in a similar way,
but can also be deduced immediately from the special unitary case via
the following lemma.

\begin{lemma}\label{T:ident-distances}
  If $U \in \Unitary{n}$ and $V \in \SUnitary{n}$, then $d_1(\mu_U,
  \nu)$ and $d_1(\mu_V, \nu)$ are identically distributed.
\end{lemma}

\begin{proof}
  Define a coupling of $U$ and $V$ as follows.  Let $V \in
  \SUnitary{n}$ be Haar-distributed, and let $\omega \in \Circle$ be
  uniformly distributed independently of $V$. Define $U = \omega V$.
    
  Now given any fixed $W \in \Unitary{n}$, $W = \xi Y$ for some $\xi
  \in \Circle$ and $Y \in \SUnitary{n}$, and thus
  \[
  UW = (\omega \xi) (VY)
  \]
  and
  \[
  WU = (\xi \omega) (YV)
  \]
  both have the same distribution as $\omega V = U$. Therefore $U \in
  \Unitary{n}$ is Haar-distributed.
  
  It follows that
  \[
  d_1(\mu_U, \nu) = d_1(\mu_{\omega V}, \nu) = d_1(\mu_V, \nu)
  \]
  since $\mu_{\omega V}$ is a translation (in $\Circle$) of $\mu_V$
  and $\nu$ is translation-invariant.
\end{proof}

An analogous statement to Lemma \ref{T:ident-distances} holds for
$\Orthogonal{n}$ and $\SOrthogonal{n}$ when $n$ is odd; in that case
$\omega, \xi \in \{-1, 1\}$ in the proof above.  When $n$ is even,
$-I_n \in \SOrthogonal{n}$ and so the argument breaks down, requiring
a different approach to deducing the main results for
$\Orthogonal{n}$.

\medskip

The next result carries out the second step in the plan of this
section.

\begin{thm}\label{T:group-entropy}
  Let $G$ be one of $\Orthogonal{n}$, $\SOrthogonal{n}$, $\SOneg{n}$,
  $\Unitary{n}$, $\SUnitary{n}$, or $\Symplectic{n}$, and let $U \in
  G$. Then
  \begin{equation}\label{E:group-entropy}
  \E d_1(\mu_U, \nu) \le C n^{-2/3}.
  \end{equation}
\end{thm}

\begin{proof}
  Assume for now that $G$ is one of $\SOrthogonal{n}$, $\SOneg{n}$,
  $\SUnitary{n}$, or $\Symplectic{n}$.

  Let $\Lip_0(\Circle) = \{f \in \Lip(\Circle) : f(1) = 0\}$, and
  observe that the Lipschitz seminorm $\abs{\cdot}_{\Lip}$ is a norm
  on this space; denote by $B(\Lip_0(\Circle))$ its unit ball. For
  $f:\Circle \to \R$, define the random variable
  \[
  X_f = \int f \ d\mu_U - \int f \ d\mu.
  \]
  Note that $\E X_f = 0$ for every $f$.  Since the value of $X_f$ is
  unchanged by adding a constant to $f$, by
  \eqref{D:distance_testfcns},
  \[
  d_1(\mu_U, \mu) = \sup \bigl\{ X_f \mid f \in B(\Lip_0(\Circle))
  \bigr\}.
  \]

  Fix $m \in \N$, to be determined later, and let $\Lip_0^m(\Circle)$
  be the $(m-1)$-dimensional subspace of $\Lip_0(\Circle)$ consisting
  of functions which, when interpreted instead as $2\pi$-periodic
  functions on $\R$, are affine on each subinterval
  $\bigl[\frac{2(k-1)\pi}{m}, \frac{2 k \pi}{m} \bigr]$ for $k \in
  \Z$.  Given $f \in B(\Lip_0(\Circle))$, there is a unique $g \in
  B(\Lip_0^m(\Circle))$ such that $g\bigl(\exp \bigl(i \frac{2 k
    \pi}{m} \bigr)\bigr) = f\bigl(\exp \bigl(i \frac{2 k \pi}{m}
  \bigr)\bigr)$ for every $k$.  Then $\norm{f - g}_\infty \le
  \frac{\pi}{m}$, so that
  \[
  \abs{X_f - X_g} \le \frac{2 \pi}{m}
  \]
  almost surely.  It follows that
  \begin{equation}\label{E:group-pre-Dudley}
  d_1(\mu_U, \mu) \le \sup \bigl\{ X_g \mid g \in B(\Lip_0^m(\Circle))
  \bigr\} + \frac{2 \pi}{m}.
  \end{equation}

  By Corollary \ref{T:group-concentration}, for $g, h \in
  B(\Lip_0^m(\Circle))$,
  \[
  \Prob \bigl[ \abs{X_g - X_h} \ge t \bigr] = \Prob \bigl[
  \abs{X_{g-h}} \ge t \bigr] \le 2 e^{-c n^2 t^2 / \abs{g-h}_{\Lip}^2}
  \]
  for every $t > 0$.  Thus by Dudley's entropy bound \cite{Du2},

  \begin{equation}\label{E:group-Dudley}
  \E \sup \bigl\{ X_g \mid g \in B(\Lip_0^m(\Circle)) \bigr\}
  \le \frac{C}{n} \int_0^\infty \sqrt{\log N(B(\Lip_0^m(\Circle)), 
    \abs{\cdot}_{\Lip}, \eps)} \ d\eps,
  \end{equation}
  where $N(B(\Lip_0^m(\Circle), \abs{\cdot}_{\Lip}, \eps)$ denotes the
  minimum number of $\eps$-balls with respect to $\abs{\cdot}_{\Lip}$
  needed to cover $B(\Lip_0^m(\Circle))$. (For a very neat exposition
  of Dudley's bound, see Section 1.2 of \cite{Ta}.) Since
  $B(\Lip_0^m(\Circle))$ is itself a ball with respect to the norm
  $\abs{\cdot}_{\Lip}$, there is the standard volumetric estimate
  \cite[Lemma 2.6]{MiSc}
  \[
  N(B(\Lip_0^m(\Circle)), \abs{\cdot}_{\Lip}, \eps)
  \le \left(\frac{3}{\eps}\right)^{m-1}.
  \]
  Inserting this into \eqref{E:group-Dudley} and then inserting the
  resulting estimate into \eqref{E:group-pre-Dudley} yields
  \[
  \E d_1(\mu_U, \mu) \le C \frac{\sqrt{m}}{n} + \frac{2 \pi}{m}.
  \]
  Picking $m$ of the order $n^{2/3}$ yields that
  \[
  \E d_1(\mu_U, \mu) \le \frac{C}{n^{2/3}},
  \]
  and so the theorem (except for the cases of $\Orthogonal{n}$ and
  $\Unitary{n}$) follows by Theorem \ref{T:group-avg-distances} and
  the triangle inequality for $d_1$.

  If $G = \Unitary{n}$, then the theorem follows from Lemma
  \ref{T:ident-distances} and the case of $\SUnitary{n}$.

  If $G = \Orthogonal{n}$, then conditionally on $\det U$, $U$ is
  Haar-distributed in either $\SOrthogonal{n}$ or $\SOneg{n}$.  Since
  \[
  \E d_1(\mu_U, \nu) = \E \bigl(\E \bigl[ d_1(\mu_U, \nu) \mid \det U
  \bigr]\bigr),
  \]
  the theorem follows from the cases of $\SOrthogonal{n}$ and
  $\SOneg{n}$.
\end{proof}

A direct union bound argument can also be used in place of
Dudley's theorem in the proof of Theorem \ref{T:group-entropy}, but
the argument given above is considerably more elegant and concise.

\medskip

The next two results complete the plan of this section.

\begin{cor}\label{T:group-concentration-distance}
  Let $G$ be one of $\Orthogonal{n}$, $\SOrthogonal{n}$, $\SOneg{n}$,
  $\Unitary{n}$, $\SUnitary{n}$, or $\Symplectic{n}$, and let $U \in
  G$. Then
  \[
  \Prob \left[ d_1(\mu_U, \nu) \ge C n^{-2/3} + t\right]
  \le e^{-c n^2 t^2}
  \]
  for every $t > 0$.
\end{cor}

\begin{proof}
  This follows immediately from Proposition \ref{T:GM} and Theorem
  \ref{T:group-entropy}, except in the cases of $\Orthogonal{n}$ and
  $\Unitary{n}$.  If $G = \Unitary{n}$, the corollary follows from
  Lemma \ref{T:ident-distances} and the case of $\SUnitary{n}$.  If
  $G = \Orthogonal{n}$, then 
  \[
  \Prob \left[ d_1(\mu_U, \nu) \ge C n^{-2/3} + t\right] = \E \left(
    \Prob \left[ d_1(\mu_U, \nu) \ge C n^{-2/3} + t \middle\vert \det
      U \right] \right)
  \]
  and the corollary follows from the cases of $\SOrthogonal{n}$ and
  $\SOneg{n}$.
\end{proof}

\begin{cor}\label{T:group-BC}
  For each $n$ let $G_n$ be one of $\Orthogonal{n}$,
  $\SOrthogonal{n}$, $\SOneg{n}$, $\Unitary{n}$, $\SUnitary{n}$, or
  $\Symplectic{n}$, and let $U_n \in G_n$. Then with probability $1$,
  \[
  d_1(\mu_{U_n}, \nu) \le C n^{-2/3}
  \]
  for all sufficiently large $n$.
\end{cor}

\begin{proof}
  Let $t = n^{-2/3}$ in Corollary \ref{T:group-concentration-distance}
  and apply the Borel--Cantelli lemma.
\end{proof}

\medskip

The main results of this section can all be extended to Dyson's
circular ensembles (for extensive discussion, see \cite{Meh}), by a
slight variation of the same methods.  The Circular Unitary Ensemble
CUE($n$) is the same as the Haar distribution on $\Unitary{n}$.  The
Circular Orthogonal Ensemble COE($n$) is distributed as $V^T V$, where
$V$ is Haar-distributed in $\Unitary{n}$.  The Circular Symplectic
Ensemble CSE($2n$) is distributed as $JV^TJ^TV$, where $V$ is
Haar-distributed in $\Unitary{2n}$ and
\[
J = \begin{bmatrix} 0 & -1 \\ 1 & 0 \\ & & 0 & -1 \\ & & 1 & 0 \\
  & & & & \ddots \\ & & & & & 0 & -1 \\ & & & & & 1 & 0 \end{bmatrix}.
\]

\begin{thm}\label{T:circular}
  Let $U$ be drawn from COE($n$) or CSE($2n$). Then
  \[
  \E \mu_U = \nu,
  \]
  \[
  \E d_1(\mu_U, \nu) \le \frac{C}{n^{2/3}},
  \]
  and
  \[
  \Prob \left[ d_1(\mu_U, \nu) \ge C n^{-2/3} + t\right]
  \le e^{-c n^2 t^2}
  \]
  for every $t > 0$.  
  
  If, for each $n$, $U_n$ is drawn from COE($n$) or CSE($2n$), then
  with probability $1$,
  \[
  d_1(\mu_{U_n}, \nu) \le C n^{-2/3}
  \]
  for all sufficiently large $n$.
\end{thm}

\begin{proof}
  For brevity the proof is given only in the case of the COE, the
  argument for the CSE being nearly identical.

  Let $U = V^T V$, where $V \in \Unitary{n}$ is Haar-distributed, and
  fix $e^{i \theta} \in \Circle$.  Then $e^{i\theta / 2} V$ is also
  Haar-distributed in $\Unitary{n}$, so $U$ has the same distribution
  as $(e^{i\theta/2} V)^T (e^{i\theta/2} V) = e^{i\theta} U$.
  Therefore $\E \mu_U$ is a rotation-invariant probability measure on
  $\Circle$, and is hence equal to $\nu$.

  Next, arguing as in the proof of Lemma \ref{T:ident-distances}, $U$
  has the same distribution as $(\omega W^T) (\omega W) = \omega^2 W^TW$,
  where $W \in \SUnitary{n}$ is Haar-distributed and $\omega \in
  \Circle$ is uniformly distributed independently of $W$.  Since
  $\omega^2$ is distributed as $\omega$, $U$ has the same distribution
  as $\omega W^T W$.  As in the proof of Lemma
  \ref{T:ident-distances}, it follows that $d_1(\mu_U, \nu)$ has the
  same distribution as $d_1(\mu_{W^T W}, \nu)$.

  Now given $W_1, W_2 \in \SUnitary{n}$,
  \begin{align*}
    \norm{W_1^T W_1 - W_2^T W_2}_{HS} & \le \norm{W_1^T (W_1 -
      W_2)}_{HS}
    + \norm{(W_1^T - W_2^T) W_2}_{HS} \\
    & = \norm{W_1 - W_2}_{HS} + \norm{W_1^T - W_2^T}_{HS} = 2
    \norm{W_1 - W_2}_{HS}.
  \end{align*}
  Thus the map $\SUnitary{n} \to \SUnitary{n}$ given by $W \mapsto W^T
  W$ is $2$-Lipschitz, and so by Proposition \ref{T:GM},
  \[
  \Prob \bigl[ F(W^T W) - \E F(W^T W) \ge t \bigr]
  \le e^{-cnt^2}
  \]
  for every $t > 0$ and every $1$-Lipschitz function $F: \SUnitary{n}
  \to \R$.

  The remainder of the proof is the same as the proofs of Theorem
  \ref{T:group-entropy}, Corollary
  \ref{T:group-concentration-distance}, and Corollary
  \ref{T:group-BC}.
\end{proof}


\bigskip

\section{Some random Hermitian matrices}\label{S:Hermitian}

In this section, we prove results comparable to Theorem
\ref{T:group-entropy} and Corollaries
\ref{T:group-concentration-distance} and \ref{T:group-BC} for two
models of Hermitian random matrices.  An essential condition on some
of the random matrices used in the constructions below is the
following.

Let $A$ be a random $n \times n$ Hermitian matrix.  Suppose that
for some $C, c > 0$, 
\begin{equation}\label{E:Hermitian-concentration}
\Prob \bigl[ \abs{F(A) - \E F(A)} \ge t \bigr]
\le C \exp \bigl[ - c n t^2 \bigr]
\end{equation}
for every $t > 0$ and $F:\mathcal{M}_n^{sa} \to \R$ which is $1$-Lipschitz with
respect to the Hilbert--Schmidt norm.  Examples in which this condition
is satisfied include:
\begin{enumerate}
\item The diagonal and upper-diagonal entries of $M$ are independent
  and each satisfy a quadratic transportation cost inequality with
  constant $c / \sqrt{n}$.  This is slightly more general than
  assuming a log-Sobolev inequality (see \cite[Section 6.2]{Ledoux}),
  and is essentially the most general condition with independent
  entries (see \cite{Go}).  It holds, e.g., for Gaussian entries and,
more generally, for entries with densities of the form $e^{-n u_{ij}(x)}$
where $u''_{ij}(x)\ge c>0.$
\item The distribution of $M$ itself has a density proportional to $e^{- n
    \tr u(M)}$ with $u:\R \to \R$ such that $u''(x) \ge c > 0$.  This
  is a subclass of the so-called unitarily invariant ensembles,
  studied extensively in mathematical physics (see \cite{DeGi}).  The
  hypothesis on $u$, via the Bakry--\'Emery theorem, guarantees that
  $M$ satisfies a log-Sobolev inequality; cf.\ \cite[Proposition
  4.4.26]{AnGuZe}.
\end{enumerate}

\medskip

One could also consider the situation in which
\eqref{E:Hermitian-concentration} is only assumed to hold for convex
Lipschitz functions $F$.  By Talagrand's theorem (see e.g.\
\cite[Section 4.2]{Ledoux}), this is the case if the diagonal and
upper-diagonal entries of $M$ are independent and supported in sets of
diameter at most $c / \sqrt{n}$.  Under this weaker condition, the
arguments below can be applied to prove results analogous to Theorem
\ref{T:group-entropy} and Corollaries
\ref{T:group-concentration-distance} and \ref{T:group-BC}, not for
$d_1(\mu_M,\E\mu_M)$ but for a ``convex-Wasserstein distance'' defined
by
\[
d_{1,\mathrm{c}}(\mu,\nu):=\sup_{\substack{f \in
    B(\Lip(\R)),\\f\,\text{convex}}}\abs{\int fd\mu-\int fd\nu}.
\]
This distance is also a metric for weak convergence of laws (see,
e.g., the proof of \cite[Theorem 2]{MMe}).

\medskip

The first model of random Hermitian matrix considered in this section
is the following.  Let $U\in\Unitary{n}$ distributed according to Haar
measure, independent of $A$, and let $P_k$ denote the projection of
$\R^n$ onto the span of the first $k$ basis elements.  Define a random
matrix $M$ by
\begin{equation}\label{E:compression}
M:=P_k U A U^* P_k^*.
\end{equation}
Then $M$ is a compression of $A$ (as an operator on $\R^n$) to a
random $k$-dimensional subspace chosen independently of $A$.  In the
case that $\{A_n\}_{n\in\N}$ is a deterministic sequence of matrices
with a limiting spectral distribution and $\frac{k}{n}\to\alpha$, the
limiting spectral distribution of $M$ can be determined using
techniques of free probability (see \cite{Sp}); the limit is given by
a free-convolution power related to the limiting spectral distribution
of $A_n$ and the value $\alpha$.  The concentration properties of the
spectral distribution of $M$ for $A$ deterministic were treated in
\cite{MeMe}, and the results below improve on those appearing in that
paper.

In the case that $k=n$, the empirical spectral measure $\mu_M$ of $M$
is the same as $\mu_A$; in particular, if $A$ satisfies a log-Sobolev
inequality and $k=n$, then the results below on the concentration of
$\mu_M$ about its mean improve on the comparable results of Guionnet
and Zeitouni from \cite{GuZe}, both in terms of the specific bounds
and in the metric used.  (The metric used in \cite{GuZe}, although
referred to there as Wasserstein, is more commonly referred to as the
bounded-Lipschitz distance and metrizes a slightly weaker topology
than the metric used here.)  We show below that the expected Wasserstein
distance of $\mu_M$ to $\E\mu_M$ is of order $n^{-2/3}$, whereas what follows
from the results of \cite{GuZe} is that the expected bounded-Lipschitz distance
of $\mu_M$ to $\E\mu_M$ is of order $n^{-2/5}$.

In the further special case that the entries on
and above the diagonal are assumed to be independent, the results
below have been surpassed (in Kolmogorov distance) in the very recent
work of G\"otze and Tikhomirov \cite{GoTi}, who proved for such matrices
that the Kolmogorov distance between the empirical spectral distribution 
 and the semicircular law is almost surely of order 
$O(n^{-1}\log^b n)$ with some positive constant $b>0$, under mild conditions
on the distributions of the entries.

The proofs below follow the same approach as described in the final
two steps of the outline given in Section \ref{S:lie-groups}.  Namely,
measure concentration, both on $\Unitary{n}$ and from the hypothesis
of \eqref{E:Hermitian-concentration}), is used together with entropy
methods to show that $\E d_1(\mu_M,\mu)$ is small, and moreover that
$d_1(\mu_M,\mu)$ is strongly concentrated near its mean.  Here,
$\mu_M$ is again the empirical spectral measure of $M$ and $\mu=\E
\mu_M$; in this section, $\mu$ is always used as a reference measure.
An additional truncation argument will be necessary, since the support
of $\mu_M$ is not necessarily uniformly bounded in this context.

The following lemma is proved using a standard discretization argument. 

\begin{lemma}[cf.\ {\cite[Proof of Proposition 4]{MeSz}}]\label{T:op-norm-bounds}
  Suppose that $\norm{\E A}_{op} \le C'$ and $A$ satisfies
  \eqref{E:Hermitian-concentration} for every convex $1$-Lipschitz
  function $F:\mathcal{M}_n^{sa} \to \R$.  Then there is a constant
  $K$ depending only on $C,c,C'$ such that
  \[
  \E \norm{A}_{op} \le K.
  \]
\end{lemma}

Observe that it follows from Lemma \ref{T:op-norm-bounds} that
$\E\norm{M}_{op}\le K$ for $M$ defined in \eqref{E:compression}.

The next preliminary lemma and corollary are needed to obtain
concentration properties for $M$ from those of $A$ and $U$.

\begin{lemma} \label{T:conj-lipschitz} Let $A \in \mathcal{M}_n^{sa}$
  be fixed. The map $\Unitary{n} \to \mathcal{M}_n^{sa}$, $U \mapsto U
  A U^*$ is $\delta(A)$-Lipschitz.
\end{lemma}

\begin{proof}
  For $\lambda \in \R$, let $A_\lambda = A - \lambda I$.  For any $U,
  V \in \Unitary{n}$,
  \begin{align*}
    \norm{U A U^* - V A V^*}_{HS} & =
    \norm{U A_\lambda U^* - V A_\lambda V^*}_{HS} \\
    & = \norm{ U A_\lambda (U - V)^* + (U - V) A_\lambda V^*}_{HS} \\
    & \le \norm{U A_\lambda (U - V)^*}_{HS} + \norm{(U - V) A_\lambda V^*}_{HS} \\
    & \le \norm{U A_\lambda}_{op} \norm{(U - V)^*}_{HS}
    + \norm{U - V}_{HS} \norm{A_\lambda V^*}_{op} \\
    & = 2 \norm{A_\lambda}_{op} \norm{U - V}_{HS}.
  \end{align*}
  Here we have used the facts that
  \begin{enumerate}
  \item $\norm{U}_{op} = 1$ for $U \in \Unitary{n}$,
  \item $\norm{AB}_{op} \le \norm{A}_{op} \norm{B}_{op}$ for $A, B \in
    \mathcal{M}_n$, and
  \item $\norm{AB}_{HS} \le \norm{A}_{op} \norm{B}_{HS}$ for $A, B \in
    \mathcal{M}_n$.
  \end{enumerate}
  Recalling that $\delta(A)=\inf_{\lambda}\norm{A_\lambda}_{op}$,
  optimizing over $\lambda$ proves the lemma.
\end{proof}

In \cite{MeMe} a weaker result is proved, essentially using instead of
the third fact above the weaker estimate $\norm{AB}_{HS} \le
\norm{A}_{HS} \norm{B}_{HS}$.

\begin{cor} \label{T:sub-lipschitz} 
  Let $A \in \mathcal{M}_n^{sa}$ be fixed and let $1 \le k \le n$.
  Then the map $\Unitary{n} \to \mathcal{M}_k^{sa}$ given by $U
  \mapsto P_k U A U^* P_k^*$ is $\delta(A)$-Lipschitz.
\end{cor}

\begin{proof}
  Combine the Lemma \ref{T:conj-lipschitz} with the obvious fact that
  $A \mapsto P_k A P_k^*$ is $1$-Lipschitz $\mathcal{M}_n^{sa} \to
  \mathcal{M}_k^{sa}$ (since $P_k A P_k^*$ is just a submatrix of
  $A$).
\end{proof}

\begin{thm}\label{T:X-concentration}
  Suppose that $A$ satisfies \eqref{E:Hermitian-concentration} for
  every $1$-Lipschitz function $F:\mathcal{M}_n^{sa} \to \R$.
  \begin{enumerate}
  \item If $F:\mathcal{M}_n^{sa} \to \R$ is $1$-Lipschitz, then for
    $M = P_k U A U^* P_k^*$,
    \[
    \Prob \bigl[ \abs{F(M) - \E F(M)} \ge t \bigr]
    \le C \exp \bigl[ - c n t^2 \bigr]
    \]
    for every $t>0$.
  \item In particular, 
    \[\Prob \bigl[ \bigl\vert \norm{M}_{op} - \E \norm{M}_{op}\bigr\vert
    \ge t \bigr]
    \le C \exp \bigl[ - c n t^2 \bigr]\]
    for every $t>0$.
  \item For any fixed probability measure $\mu \in \mathcal{P}_2(\C)$
    and $1$-Lipschitz $f:\R \to \R$, if
    \[
    X_f = \int f \ d\mu_M - \int f \ d\mu,
    \]
    then
    \[
    \Prob \bigl[\abs{X_f - \E X_f} \ge t \bigr]
    \le C e^{- c kn t^2}
    \]
    for every $t > 0$.
  \item For any fixed probability measure $\mu \in \mathcal{P}_2(\C)$
    and $1\le p \le 2$,
    \[
    \Prob \bigl[\abs{d_p(\mu_M,\mu) - \E d_p(\mu_M,\mu)} \ge t \bigr]
    \le C e^{-c kn t^2}
    \]
    for every $t > 0$.
  \end{enumerate}
\end{thm}

\begin{proof}
For the first part, observe that 
\begin{equation*}\begin{split}
\Prob \bigl[ \abs{F(M) - \E F(M)} \ge t \bigr]&\le
\E\left(\Prob \left[\left. \abs{F(M) - \E \left[F(M) 
\middle\vert U\right]} 
\ge \frac{t}{2}\right| U \right]\right)
\\&\qquad\qquad
+\Prob \left[ \abs{\E \left[F(M) \middle\vert U\right]- \E F(M)} \ge 
\frac{t}{2} \right].
\end{split}\end{equation*}
Conditional on $U$, $F(M)$ is a $1$-Lipschitz function of $A$, and by
taking expectation over $U$ in Corollary \ref{T:sub-lipschitz}, it
follows that $\E\left[F(M) \middle\vert U\right]$ is an
$\E[\delta(A)]$-Lipschitz function of $U$.  The first part thus
follows from the hypothesis on $A$ and Lemma \ref{T:op-norm-bounds}.
Part (2) follows from part (1) and the fact that the operator norm is
a 1-Lipschitz function with respect to the Hilbert--Schmidt norm on
$\mathcal{M}_n^{sa}$.  The remaining parts follow from Lemma
\ref{T:Lipschitz-estimates} and part (1).
\end{proof}

To estimate $\E d_1(\mu_M, \mu)$ (where, as before, $\mu = \E \mu_M$)
the arguments in the previous section can be supplemented with a
truncation argument using the lemma above to obtain the following.

\begin{thm}\label{T:compression-entropy}
  Suppose that $A$ satisfies \eqref{E:Hermitian-concentration} for
  every $1$-Lipschitz function $F:\mathcal{M}_n^{sa} \to \R$.  Let $M
  = P_k U A U^* P_k^*$, and let $\mu_M$ denote the empirical spectral
  distribution of $M$ with $\mu=\E \mu_M$.  Then
  \[
  \E d_1(\mu_M,\mu) \le \frac{C''(\E\norm{M}_{op})^{1/3}}{(kn)^{1/3}}
  \le \frac{C'''}{(kn)^{1/3}},
  \]
  and so
  \[
  \mathbb{P}\left[d_1(\mu_M,\mu) > \frac{C'''}{(kn)^{1/3}}+
    t\right]\le Ce^{-cknt^2}
  \]
  for each $t > 0$.
\end{thm}

\begin{proof}
  Denote by $\Lip_0(\R) = \{ f \in \Lip(\R) : f(0) = 0\}$, and observe
  that
  \[
  \E d_1(\mu_M,\mu) = \E \sup\left\{ X_f : f \in B(\Lip_0(R)) \right\},
  \]
  where 
  \[
  X_f := \int f \ d\mu_M-\int f \ d\mu
  \] 
  as before.  The indexing space can be reduced to compactly supported
  functions via a truncation argument, as follows.  Fix $R>0$, and let
  \[
  f_R(x) = \begin{cases} f(x) & \text{ if } \abs{x} \le R; \\
    f(R) + \bigl[\sgn(f(R))\bigr](R-x) & \text{ if } R < x < R + \abs{f(R)})\\
    f(-R)+\bigl[\sgn(f(-R))\bigr](x-R) & \text{ if } -\abs{f(-R)} - R < x < -R;\\
    0 & \text{ if } x \le -R - \abs{f(-R)}] \text{ or } x \ge R +
    \abs{f(R)};
  \end{cases}
  \]
  that is, $f_R=f$ for $\abs{x}\le R$ and then drops off linearly to
  zero, so that $f_R$ is 1-Lipschitz, $f(x)=0$ for $\abs{x} > 2R$, and
  $\abs{f(x)-f_R(x)} \le \abs{x}$ for all $x \in \R$.  Then by
  Fubini's theorem,
  \begin{equation*}\begin{split}
      \abs{\int f \ d\mu_M-\int f_R \ d\mu_M}&
      \le \int_{\abs{x} > 2R} \abs{x} \ d\mu_M(x) \\
      & \le 2R \int_{\abs{x} > 2R} \ d\mu_M(x) 
      + \int_{2R}^\infty \mu_M((t,\infty)) \ dt
      + \int_{-\infty}^{-2R}\mu_M((-\infty,t)) \ dt.
    \end{split}\end{equation*}
  Taking the supremum over $f$ followed by expectation over $M$, and
  making use of part (2) of Theorem \ref{T:X-concentration} together
  with the trivial bound
  $\E\mu_M\bigl((-\infty,t)\cup(t,\infty)\bigr)\le n
  \P[\norm{M}_{op}\ge t]$ yields
  \[
  \E\sup\left\{\abs{\int (f-f_R) \ d\mu_M} : f \in B(\Lip_0(\R)) 
  \right\}\le
  CRn\exp\left[-cn(2R-\E\norm{M}_{op})^2 \right],
  \]
  and the same holds if $\mu_M$ is replaced by $\mu$.  Taking, for
  example, $2R=\E\norm{M}_{op}+1$ gives that
  \[
  \E\sup\left\{\abs{X_f - X_{f_R}} : f \in B(\Lip_0(\R)) \right\}\le
  Cn\bigl(\E\norm{M}_{op}\bigr)e^{-cn}.
  \] 
  Consider therefore the process $X_f$ indexed by
  $\Lip_{1,\frac{1}{2}(\E\norm{M}_{op}+1)}$ (with norm $\abs{\cdot}_{\Lip}$),
  where
  \[
  \Lip_{a,b} := \left\{f:\R\to\R : \abs{f}_{\Lip}\le a; f(x)=0 \text{ if } 
    \abs{x} > b \right\}.
  \]
  The above argument shows that
  \begin{equation}\label{trunc_error}
    \E\Bigl[d_1(\mu_M,\mu) \Bigr] \le 
    \E\Bigl[\sup\left\{X_f : f \in \Lip_{1,\frac{1}{2}
        (\E\norm{M}_{op}+1)} \right\} \Bigr]
    + C n \bigl(\E\norm{M}_{op}\bigr)e^{-cn}.
  \end{equation}
  
  Now that the indexing space of the process has been reduced to
  compactly supported functions, the proof can be completed exactly as
  in the case of Theorem \ref{T:group-entropy}; the additional error
  incurred by the truncation above is negligible compared to the
  errors produced by the earlier argument.  The factor
  $(\E\norm{M}_{op})^{1/3}$ in the final bound is due to the size of
  the truncation parameter $R$ (in the proof of Theorem
  \ref{T:group-entropy}, the corresponding quantity was simply $2\pi$
  and therefore disappeared into the constants in the statement).
\end{proof}

\begin{cor}
  For each $n$, let $A_n \in \mathcal{M}_n^{sa}$ be fixed with
  spectrum bounded independently of $n$.  Let $U_n \in \Unitary{n}$ be
  Haar-distributed and fix $k$. Let $M_n = P_k U A_n U^* P_k^*$ and
  let $\mu_n = \E \mu_{M_n}$. Then with probability $1$,
  \[
  d_1(\mu_{M_n}, \mu_n) \le C n^{-1/3},
  \]
  where $C$ depends only on $k$ and the bounds on the sizes of the
  spectra of $A_n$.
\end{cor}

\begin{proof}
  This follows from Theorem \ref{T:compression-entropy}, using $t =
  n^{-1/3}$ and the Borell--Cantelli lemma.
\end{proof}

\bigskip

The second model of random matrix considered in this section is is
defined as follows.  Let $A,B\in\mathcal{M}_n^{sa}$ satisfy condition
\eqref{E:Hermitian-concentration} let $U \in \Unitary{n}$ be Haar
distributed, with $A,B,U$ independent.  Define \[M = U A U^* + B,\]
the ``randomized sum'' of $A$ and $B$.  In the case of deterministic
sequences $\{A_n\}$ and $\{B_n\}$, this model has been studied at some
length.  The limiting spectral measure was studied first by Voiculescu
\cite{Vo} and Speicher \cite{Sp2}, who showed that if $\{A_n\}$ and
$\{B_n\}$ have limiting eigenvalue distributions $\mu_A$ and $\mu_B$
respectively, and if $M_n:=UA_nU^*+B_n$, then the limiting spectral
distribution of $M_n$ is given by the free convolution
$\mu_A\boxplus\mu_B$.  More recently, Chatterjee \cite{Ch} showed
subexponential concentration (up to a logarithmic factor) of
$\mu_{M_n}$ about its mean; Kargin \cite{Kargin} improved this to
subgaussian concentration (again up to a logarithmic factor), and was
furthermore able to consider the distance to
$\mu_{A_n}\boxplus\mu_{B_n}$ itself, rather than $\E\mu_{M_n}$.
Theorem \ref{T:AB-entropy} below gives a similar level of
concentration to Kargin's result.  The main differences are that here
the reference measure is $\E\mu_{M_n}$ rather than a free convolution;
the matrices $A_n$ and $B_n$ may be random here, whereas Kargin's
result requires $A_n$ and $B_n$ to be deterministic; and Kargin's
result is in terms of Kolmogorov distance, rather than Wasserstein
distance.

The proofs below once again follow the same approach as described in
the final two steps of the outline given in Section
\ref{S:lie-groups}.

Note that by Weyl's inequalities \cite[Theorem III.2.1]{Bhatia}, the
spectrum of $M$ always lies in the interval $[\lambda_{\min}(A) +
\lambda_{\min}(B), \lambda_{\max}(A) + \lambda_{\max}(B)]$, of length
$\delta(A) + \delta(B)$, and so by Lemma \ref{T:op-norm-bounds},
$\E\norm{M}_{op}$ is bounded in terms of the constants in
\eqref{E:Hermitian-concentration} for $A$ and $B$.  We also have the
following analog of Theorem \ref{T:X-concentration}.

\begin{thm}[cf.\ {\cite[Corollary 4.4.30]{AnGuZe}}] \label{T:AB} Let
  $A,B \in \mathcal{M}_n^{sa}$ satisfying
  \eqref{E:Hermitian-concentration} and let $U \in \Unitary{n}$ be
  Haar-distributed with $A,B,U$ independent.  Define $M = U A U^*+ B$.
  \begin{enumerate}
  \item There exist $C,c$ depending only on the constants in
    \eqref{E:Hermitian-concentration} for $A$ and $B$, such that if
    $F: \mathcal{M}_k^{sa} \to \R$ is $1$-Lipschitz, then
    \[
    \Prob \bigl[\abs{F(M) - \E F(M)} \ge t \bigr]
    \le C \exp \left[- cnt^2\right]
    \]
    for every $t > 0$.
  \item In particular, 
    \[
    \Prob \bigl[\abs{\norm{M}_{op} - \E \norm{M}_{op}} \ge t \bigr]
    \le C \exp \left[- cnt^2\right]
    \]
    for every $t > 0$.
  \item For any fixed probability measure $\rho \in \mathcal{P}_2(\R)$
    and $1$-Lipschitz $f:\R \to \R$, let 
    \[
    X_f = \int f \ d\mu_M - \int f \ d\rho.
    \]
    Then
    \[
    \Prob \bigl[\abs{X_f - \E X_f} \ge t \bigr]
    \le C \exp \left[ - cn^2 t^2\right]
    \]
    for every $t > 0$.
  \item For any fixed probability measure $\rho \in \mathcal{P}_2(\R)$
    and $1\le p \le 2$,
    \[
    \Prob \bigl[\abs{d_p(\mu_M,\mu) - \E d_p(\mu_M,\mu)} \ge t \bigr]
    \le C \exp \left[ - c n^2 t^2\right]
    \]
    for every $t > 0$.
  \end{enumerate}
\end{thm}

\begin{proof}
  \begin{enumerate}
  \item 
    By the coupling described in the proof of 
    Lemma \ref{T:ident-distances}, we may equivalently define
    \[
    M = (\omega V) A (\omega V)^*  + B
    =  V A V^*  + B
    \]
    for $\omega$ and $V$ independent with $\omega$ uniformly distributed
    in $\Circle$ and $V$ Haar-distributed in $\SUnitary{n}$.  
    Now, 
    \begin{align*}
      \Prob\bigl[\abs{F(M) - \E F(M)} \ge t \bigr] \le &
      \E\left(\Prob\left[\left.\abs{F(M)-\E\left[F(M) \middle\vert
                A,V\right]} \ge \frac{t}{3}
          \right|A,V\right]\right)\\
      &+\E\left( \Prob\left[\left.\abs{\E\left[F(M) \middle\vert
                A,V\right] -\E\left[F(M) \middle\vert V\right]} \ge
            \frac{t}{3}
          \right|V\right]\right)\\
      &+\Prob\left[\abs{\E\left[F(M) \middle\vert V\right] -\E F(M)}
        \ge \frac{t}{3}\right].
      \end{align*}
    Conditional on $A$ and $V$, $F(M)$ is a $1$-Lipschitz function of
    $B$, and by independence, the distribution of $B$ is unchanged by
    conditioning on $A$ and $V$.  The conditional distribution of $B$
    therefore still satisfies the concentration hypothesis and so the
    first summand above is bounded as desired.  Similarly, conditional
    on $V$, $\E\left[F(M) \middle\vert A,V\right]$ is a $1$-Lipschitz
    function of $A$, and the bound on the second summand follows from
    independence and the concentration hypothesis for $A$.  By
    Corollary \ref{T:sub-lipschitz}, $M$ is $\delta(A)$-Lipschitz as a
    function of $V$; it follows that $\E\left[F(M) \middle\vert
      V\right]$ is an $\E[\delta(A)]$-Lipschitz function of $V$, and
    the claim then follows from Lemma \ref{T:op-norm-bounds} and
    Proposition \ref{T:GM}.
    
  \item This follows from the previous part and the fact that the
    operator norm is a 1-Lipschitz function with respect to the
    Hilbert-Schmidt norm on $\mathcal{M}_n^{sa}$.
    
  \item As a function of $\mu_M \in \mathcal{P}_1(\R)$, $X_f$ is
    $1$-Lipschitz by the duality between $d_1$ and $1$-Lipschitz
    functions on $\R$.  By Lemma \ref{T:Lipschitz-estimates}, $\mu_M$
    is $n^{-1/2}$-Lipschitz as a function of $M$, and so the claim
    follows from the first part.
    
  \item This also follows from the first part and Lemma
    \ref{T:Lipschitz-estimates}.
  \end{enumerate}
\end{proof}

\begin{thm}\label{T:AB-entropy}
  In the setting of Theorem \ref{T:AB}, there are constants
  $c,C,C',C''$ depending only on the concentration hypotheses for $A$
  and $B$, such that
  \[
  \E d_1(\mu_M, \mu) \le \frac{C(\E\norm{M}_{op})^{1/3}}{n^{2/3}}
  \le \frac{C'}{n^{2/3}},
  \]
  and so
  \[
  \Prob \left[ d_1(\mu_M, \mu) \ge \frac{C'}{n^{2/3}} + t \right] \le
  C''e^{ - c n^2 t^2}
  \]
  for $t > 0$.
\end{thm}

The proof is exactly the same as the proof of Theorem 
\ref{T:compression-entropy}.

\begin{cor}
  For each $n$, let $A_n, B_n \in \mathcal{M}_n^{sa}$ be fixed
  matrices with spectra bounded independently of $n$. Let $U_n \in
  \Unitary{n}$ be Haar-distributed. Let $M_n = UA_nU^* + B_n$ and let
  $\mu_n = \E \mu_{M_n}$. Then with probability $1$,
  \[
  d_1(\mu_{M_n}, \mu_n) \le C n^{-2/3}
  \]
  for all sufficiently large $n$, where $C$ depends only on the bounds
  on the sizes of the spectra of $A_n$ and $B_n$.
\end{cor}

\begin{proof}
  This follows from Theorem \ref{T:AB-entropy}, using $t = n^{-2/3}$
  and the Borel--Cantelli lemma.
\end{proof}

\bigskip

\section*{Acknowledgements}
E.\ Meckes's research is partially supported by a Five-Year Fellowship
from the American Institute of Mathematics and NSF grant
DMS-0852898. M.\ Meckes's research is partially supported by NSF grant
DMS-0902203.  M.\ Meckes thanks the Mathematisches Forschungsinstitut
Oberwolfach, where part of this research was carried out.

\bibliographystyle{plain}
\bibliography{ccrsmrm}

\end{document}